\documentclass[thmsec,dntnbr,nolinenbrs]{dmgt}

\usepackage{amsmath}
\usepackage{enumitem}

\allowdisplaybreaks
\numberwithin{equation}{section}

\newcommand{\F}{\mathbb{F}}
\newcommand{\Z}{\mathbb{Z}}
\newcommand{\one}{\mathbf{1}}
\newcommand{\Span}{\operatorname{span}}
\newcommand{\nd}{\operatorname{nd}}

\newauthor{Arthur F. Ramos}{A. F. Ramos}{Microsoft, Redmond, WA, USA}[arfreita@microsoft.com]
\newauthor{David Barros Hulak}{D. B. Hulak}{Independent researcher}[dbhulak@gmail.com]
\newauthor{Ruy J. G. B. de Queiroz}{R. J. G. B. de Queiroz}{Centro de Inform\'atica\\
Universidade Federal de Pernambuco\\
Recife, Brazil}[ruy@cin.ufpe.br]

\title{Pair-Trace Absorption Certificates for Regular Induced Subgraphs}[Pair-Trace Certificates]


\classnbr{05C35, 05C55, 05C50}

\begin{document}

\begin{abstract}
We study a fixed-core absorption problem for regular induced subgraphs.  A set is \(q\)-modular if all induced degrees are congruent modulo \(q\).  Given a \(q\)-modular witness \(A\) and a retained core \(U\subseteq A\), we ask when deleting equal-trace \(q\)-tuples from \(A\setminus U\) synchronizes the degrees on \(U\) modulo \(2q\).

The main contribution is a finite absorption-or-obstruction certificate.  We give an exact quotient formula for the deletion-tail obstruction: in complement-orbit coordinates the correct expression uses oriented differences \(n_B-n_{U\setminus B}\), not sums.  Equal-trace \(q\)-tuples absorb exactly the span of their trace classes in \(\F_2^U/\langle\one_U\rangle\).  In particular, a connected graph of \(q\)-heavy two-point traces on \(U\), together with one odd trace when \(|U|\) is even, absorbs every top-bit defect by deleting at most \(q(|U|-1)\) tail vertices.  If fixed-core absorption fails, the obstruction is an explicit even parity cut of \(U\).

We also record the parity base, the terminal modular criterion, and a conditional modular-witness threshold theorem explaining the relevance to the Erd\H{o}s--Fajtlowicz--Staton problem.  The paper does not claim to solve that problem or to improve the general lower bound for \(F(n)\).
\end{abstract}

\section{Introduction}

All graphs in this paper are finite, simple, and undirected.  For a graph \(G\) and a set \(S\subseteq V(G)\), write \(G[S]\) for the induced subgraph on \(S\), and write \(\deg_S(v)\) for the degree of \(v\) inside \(G[S]\).  Define
\[
        f(G)=\max\{|S|:G[S]\text{ is regular}\},
        \qquad
        F(n)=\min_{|V(G)|=n} f(G).
\]
Equivalently, \(F(n)\) is the largest integer \(k\) such that every \(n\)-vertex graph contains a regular induced subgraph on at least \(k\) vertices.  The Erd\H{o}s--Fajtlowicz--Staton problem asks whether
\[
        \frac{F(n)}{\log n}\to\infty .
\]
The problem appears in work of Erd\H{o}s and in the finite Ramsey-number formulation studied by Fajtlowicz, McColgan, Reid, and Staton \cite{Erdos1993,FMRS1995}; it is also discussed in \cite{ChungGraham1998}.  See also the later work on nearly regular induced subgraphs \cite{AKS2008}, repeated degrees \cite{CaroYuster2020}, random graphs \cite{KrivelevichSudakovWormald2011}, and recent regular-induced Ramsey numbers \cite{DysonMcKay2026}.  The classical Ramsey argument gives \(F(n)\ge c\log n\), since every graph contains a clique or an independent set of logarithmic order.  The exact regular induced version is much more rigid than approximate regularity or repeated-degree variants: deleting a vertex changes the degree of every surviving neighbor.

The present paper is not a new Ramsey-number computation, a result on nearly regular induced subgraphs, or a repeated-degree theorem.  Its object is a modular obstruction calculus for lifting induced degree congruences on a fixed core.  The Ramsey-type threshold is included only to explain why such local lifting certificates are relevant.

This paper develops a modular approach to the problem.  The guiding idea is to first find a large induced set whose degrees are equal modulo \(2\), then lift the congruence to modulo \(4\), then to modulo \(8\), and so on.  Once a set \(U\) of order at most \(q\) has all internal degrees congruent modulo \(q\), the induced graph \(G[U]\) is genuinely regular.

The main purpose of the paper is to give a graph-theoretic certificate for one fixed-core dyadic absorption step.  We also isolate the exact kind of lifting theorem that would imply a positive answer to the Erd\H{o}s--Fajtlowicz--Staton problem, but the paper's proved results are the local obstruction and absorption theorems.

The resulting fixed-core mechanism has a simple three-part form.  First, a proposed dyadic lift has an exact quotient obstruction.  Second, a connected \(q\)-heavy pair-trace reservoir forces the rank-rich condition and hence absorbs every top-bit defect.  Third, if equal-trace absorption fails for the fixed core, then failure is witnessed by an explicit even parity cut.

\subsection{Main results}

The main results of the paper are as follows.

\begin{itemize}[leftmargin=*]
    \item \textbf{Conditional threshold theorem.}  If the dyadic lift from \(2\) to \(4\), and the lifts from \(q\) to \(2q\) for \(q\ge 4\), can be carried out with polynomial loss and exact modular-witness output size, then \(F(n)/\log n\to\infty\).  This is Theorem~\ref{thm:conditional-threshold}.
    \item \textbf{Exact tail obstruction.}  The deletion-tail obstruction is the quotient class of \(\rho_R\) modulo constant functions.  In complement-orbit coordinates it is controlled by oriented differences \(n_B-n_{U\setminus B}\), not by sums.  This is Proposition~\ref{prop:complement-difference} and Corollary~\ref{cor:next-bit}.
    \item \textbf{Complete twin-block absorption criterion.}  For a fixed core \(U\), equal-trace \(q\)-tuples can correct precisely the span of their trace classes in \(\F_2^U/\langle \one_U\rangle\).  This gives an if-and-only-if linear-algebra criterion for one absorption step.  This is Theorem~\ref{thm:trace-multiplicity}.
    \item \textbf{Rank-rich and structured reservoir absorption.}  If the traces occurring at multiplicity at least \(q\) span \(\F_2^U/\langle \one_U\rangle\), then every top-bit defect on \(U\) can be synchronized by deleting at most \(q(|U|-1)\) tail vertices.  A connected \(q\)-heavy pair-trace graph, together with one odd trace when \(|U|\) is even, is a concrete graph-theoretic sufficient condition for this rank-rich hypothesis.  A separate random trace-reservoir theorem, with an explicit uniform-distribution bound, gives a model/certification version.  These are Theorems~\ref{thm:rank-rich}, \ref{thm:pair-trace-connected}, \ref{thm:random-reservoir}, and Corollary~\ref{cor:uniform-random}.
    \item \textbf{Dual obstruction certificate.}  Failure of twin-block absorption is equivalent to the existence of a linear functional that annihilates all available trace classes but detects the top-bit defect.  Equivalently, there is an even subset \(Y\subseteq U\) meeting every available trace in even parity while meeting the top-bit defect in odd parity.  The same linear algebra gives an algorithm that outputs either a deletion certificate or this parity cut.  This is Corollary~\ref{cor:dual-certificate} and Theorem~\ref{thm:parity-cut}.
    \item \textbf{Calibration classes.}  Perfect graphs contain regular induced subgraphs of order at least \(\sqrt n\), and graphs of neighborhood diversity at most \(t\) contain regular induced subgraphs of order at least \(n/t\).  These elementary calibrations show that hard examples must have neither large homogeneous sets nor low trace complexity.
\end{itemize}

\section{Modular witnesses}

\begin{definition}
For an integer \(q\ge 1\), a set \(U\subseteq V(G)\) is called \emph{\(q\)-modular} if
\[
        \deg_U(u)\equiv \deg_U(v)\pmod q
\]
for all \(u,v\in U\).
\end{definition}

\begin{lemma}[Terminal modular criterion]\label{lem:terminal}
If \(U\) is \(q\)-modular and \(|U|\le q\), then \(G[U]\) is regular.
\end{lemma}

\begin{proof}
Every degree in \(G[U]\) lies in the interval \([0,|U|-1]\), whose length is at most \(q-1\).  Two integers in this interval that are congruent modulo \(q\) must be equal.  Hence all degrees in \(G[U]\) are equal.
\end{proof}

\section{The parity base}

The dyadic program starts from the following standard parity fact.

\begin{lemma}[Parity partition]\label{lem:parity}
Every graph \(G\) on \(n\) vertices contains an induced subgraph on at least \(n/2\) vertices in which all degrees are even.
\end{lemma}

\begin{proof}
Work over \(\F_2\).  Let \(A\) be the adjacency matrix of \(G\), let \(D\) be the diagonal matrix with \(D_{vv}=\deg_G(v)\pmod 2\), and put \(L=A+D\).  This is the graph Laplacian modulo \(2\).  Let \(d\) be the vector of degrees modulo \(2\).

We claim that \(d\) lies in the column space of \(L\).  Since \(L\) is symmetric, it is enough to show that \(x\cdot d=0\) for every \(x\in\ker L\).  For such \(x\),
\[
        0=x^T Lx
        =\sum_v d_vx_v^2+\sum_{uv\in E(G)}2x_ux_v
        =\sum_v d_vx_v
        =x\cdot d
\]
in \(\F_2\).  Hence there is a vector \(c\in \F_2^{V(G)}\) with \(Lc=d\).

Let \(V_i=\{v:c_v=i\}\).  For a vertex \(v\), the parity of its degree inside its own color class is
\[
        \sum_{u\sim v}(1+c_u+c_v)
        =d_v+\sum_{u\sim v}c_u+d_vc_v
        =d_v+(Lc)_v
        =0.
\]
Thus both parts induce graphs of even degree, and one part has order at least \(n/2\).
\end{proof}

\section{A conditional dyadic threshold theorem}

For a power of two \(q\) and an integer \(m\ge 1\), let \(M(q,m)\) be the least \(N\) such that every graph with at least \(N\) vertices contains a \(q\)-modular witness of exactly \(m\) vertices.  This modular-witness threshold is finite by Ramsey's theorem, since a clique or an independent set of order \(m\) is \(q\)-modular.  It is distinct from the at-least-\(k\) threshold implicit in \(F(n)\).  The next theorem records the formal consequence of a successful dyadic lifting theory.  The exact output size belongs to the modular witness, because the terminal modular criterion applies only when the final witness has order at most its modulus.

\begin{theorem}[Conditional dyadic threshold]\label{thm:conditional-threshold}\mbox{}\par\nobreak
Suppose there are constants \(C,a,C_0>0\) such that the following two lifting statements hold.
\begin{enumerate}[label=\textup{(\arabic*)}, leftmargin=*]
    \item Every \(2\)-modular witness of size at least \(C_0m\) contains a \(4\)-modular witness of exactly \(m\) vertices.
    \item For every power of two \(q=2^j\) with \(q\ge 4\), every \(q\)-modular witness of size at least \(Cq^a m\) contains a \(2q\)-modular witness of exactly \(m\) vertices.
\end{enumerate}
Then there exists a constant \(A\) such that
\[
        M(2^r,2^r)\le 2^{A(r+1)^2}
\]
for all \(r\ge 1\).  Consequently \(F(n)/\log n\to\infty\).
\end{theorem}

\begin{proof}
It is enough to estimate the number of vertices needed to guarantee a \(2^r\)-modular witness of exactly \(2^r\) vertices.  We work backwards along the dyadic ladder.

By Lemma~\ref{lem:parity}, every graph on \(2N\) vertices contains a \(2\)-modular witness on at least \(N\) vertices.  Applying the first lift with \(m\) equal to the required \(4\)-modular target costs a factor \(C_0\).  For each later step, the lift from \(2^j\) to \(2^{j+1}\) costs at most \(C(2^j)^a\).  Therefore it is enough to start with at most
\[
        2C_0\left(\prod_{j=2}^{r-1} C(2^j)^a\right)2^r
\]
vertices, with the usual convention that the empty product is \(1\).  Taking base-two logarithms gives \(O(r)+a\sum_{j=2}^{r-1}j=O(r^2)\).  Hence, after increasing \(A\) to cover small \(r\), we obtain \(M(2^r,2^r)\le 2^{A(r+1)^2}\).

This witness has order \(2^r\) and modulus \(2^r\), so Lemma~\ref{lem:terminal} makes it regular.

Finally, choose \(r\) maximal with \(2^{A(r+1)^2}\le n\).  Then \(r\) is asymptotic to a positive constant times \(\sqrt{\log n}\), and \(2^r/\log n\to\infty\).  By the bound on \(M(2^r,2^r)\), every \(n\)-vertex graph contains a \(2^r\)-modular witness of exactly \(2^r\) vertices; by the previous paragraph this witness is regular.  Hence \(F(n)\ge 2^r\), and therefore \(F(n)/\log n\to\infty\).
\end{proof}

\section{Tails and lifting obstructions}

Let \(A\) be a \(q\)-modular witness, and let \(W\subseteq A\).  Put \(R=A\setminus W\), called the deletion tail.  For \(v\in W\), define
\[
        \rho_R(v)=|N(v)\cap R|.
\]
Since
\[
        \deg_W(v)=\deg_A(v)-\rho_R(v),
\]
the obstruction to making \(W\) modular at a higher modulus is encoded by the tail function \(\rho_R\).

Suppose \(q=2^j\) and \(A\) is \(q\)-modular.  Choose a lift \(d\) modulo \(2q\) of the common degree residue modulo \(q\).  Define the top-bit label \(b_A(v)\in\F_2\) by
\[
        \deg_A(v)\equiv d+qb_A(v)\pmod{2q}.
\]
Changing the lift \(d\) changes \(b_A\) by a constant function, so all quotient classes used below are independent of this choice.

\begin{lemma}[Affine lift criterion]\label{lem:affine-lift}
Let \(W\subseteq A\), and let \(R=A\setminus W\).  Then \(W\) is \(2q\)-modular if and only if the function
\[
        v\longmapsto \rho_R(v)-qb_A(v)
\]
is constant modulo \(2q\) on \(W\).
\end{lemma}

\begin{proof}
For \(v\in W\),
\[
        \deg_W(v)
        =\deg_A(v)-\rho_R(v)
        \equiv d+qb_A(v)-\rho_R(v)\pmod{2q}.
\]
The right side is independent of \(v\) precisely when \(\rho_R(v)-qb_A(v)\) is independent of \(v\) modulo \(2q\).
\end{proof}

This criterion is the clean algebraic core of the framework.  It suggests looking for large subsets \(W\) on which the tail function, corrected by the top-bit label, is constant modulo \(2q\).

\section{The exact tail obstruction}

The affine criterion of Lemma~\ref{lem:affine-lift} shows that the obstruction to a dyadic lift is a class of a tail-counting function modulo constants.  This section gives an exact form of that class.  It also explains why naive complement-orbit aggregation is insufficient.

Let \(U\) be a finite vertex set and let \(R\) be a tail.  For each trace \(B\subseteq U\), define
\[
        n_B=|\{x\in R:N(x)\cap U=B\}|.
\]
The tail-counting function is
\[
        \rho_R=\sum_{B\subseteq U} n_B\one_B\in \Z^U.
\]
Since \(\one_{U\setminus B}=\one_U-\one_B\), the class of \(\rho_R\) modulo constants is controlled by oriented differences between complementary traces, not by their sums.

\begin{proposition}[Exact complement-difference formula]\label{prop:complement-difference}
Choose one representative \(B\) from each complement orbit \(\{B,U\setminus B\}\), excluding the constant orbit \(\{\varnothing,U\}\).  Then, in \(\Z^U/\Z\one_U\),
\[
        [\rho_R]
        =\sum_{[B]}(n_B-n_{U\setminus B})[\one_B].
\]
Consequently, for every integer \(s\ge 1\), the class of \(\rho_R\) modulo constants in \((\Z/2^s\Z)^U/\langle \one_U\rangle\) is represented by the same expression reduced modulo \(2^s\).
\end{proposition}

\begin{proof}
Pair the two complementary trace classes in each orbit.  Their contribution to \(\rho_R\) is
\[
        n_B\one_B+n_{U\setminus B}\one_{U\setminus B}.
\]
Using \(\one_{U\setminus B}=\one_U-\one_B\), this becomes
\[
        n_{U\setminus B}\one_U+(n_B-n_{U\setminus B})\one_B.
\]
The first term is constant on \(U\), so it vanishes in the quotient by constants.  Summing over complement orbits proves the formula.  Reduction modulo \(2^s\) gives the final statement.
\end{proof}

\begin{corollary}[Exact next-bit obstruction]\label{cor:next-bit}
Let \(m\ge 0\).  Suppose \(\rho_R\) is constant modulo \(2^m\) on \(U\); equivalently, \([\rho_R]\) vanishes in \((\Z/2^m\Z)^U/\langle \one_U\rangle\).  Choose an integer \(c\) such that
\[
        \rho_R-c\one_U\in 2^m\Z^U.
\]
Define
\[
        \Theta_m(R,U)
        =\left[\frac{\rho_R-c\one_U}{2^m}\pmod 2\right]
        \in \F_2^U/\langle \one_U\rangle .
\]
Then \(\Theta_m(R,U)\) is independent of the choice of \(c\), and \(\Theta_m(R,U)=0\) if and only if \(\rho_R\) is constant modulo \(2^{m+1}\) on \(U\).
\end{corollary}

\begin{proof}
If \(c\) is replaced by another valid choice \(c'\), then \(c-c'\) is divisible by \(2^m\).  After division by \(2^m\), the change is a constant vector, which vanishes in the quotient \(\F_2^U/\langle \one_U\rangle\).  Thus \(\Theta_m\) is well defined.

The class \(\Theta_m\) vanishes exactly when \((\rho_R-c\one_U)/2^m\) is constant modulo \(2\).  This is equivalent to saying that for some integer \(c'\), \(\rho_R-c'\one_U\in 2^{m+1}\Z^U\), which is exactly constancy modulo \(2^{m+1}\).
\end{proof}

\begin{corollary}[Oriented orbit form under divisibility]\label{cor:oriented-divisibility}
Assume, in addition, that for every chosen complement-orbit representative \(B\), the difference \(n_B-n_{U\setminus B}\) is divisible by \(2^m\).  Then
\[
        \Theta_m(R,U)
        =\sum_{[B]}\left(\frac{n_B-n_{U\setminus B}}{2^m}\pmod 2\right)[\one_B]
\]
in \(\F_2^U/\langle \one_U\rangle\).
\end{corollary}

\begin{proof}
By Proposition~\ref{prop:complement-difference}, the tail class modulo constants is represented by \(\sum_{[B]}(n_B-n_{U\setminus B})[\one_B]\).  The assumed divisibility permits termwise division by \(2^m\).  Reducing the resulting coefficients modulo \(2\) gives the displayed formula.
\end{proof}

\begin{example}[Why complement sums fail]\label{ex:sums-fail}
Let \(U=\{1,2,3,4\}\), and let \(R=\{x,y\}\) with
\[
        N(x)\cap U=\{1\},
        \qquad
        N(y)\cap U=\{2,3,4\}.
\]
Then \(\rho_R=(1,1,1,1)\), so the tail function is already constant.  Hence every next-bit obstruction \(\Theta_m\) that is defined must vanish.

However, the complement orbit \(\{\{1\},\{2,3,4\}\}\) has total multiplicity
\[
        n_{\{1\}}+n_{\{2,3,4\}}=2.
\]
Dividing this sum by \(2\) and reducing modulo \(2\) gives a nonzero coefficient, even though the actual obstruction is zero.  The reason is that the two complementary traces cancel modulo constants: their oriented difference is \(n_{\{1\}}-n_{\{2,3,4\}}=0\).
\end{example}

Thus any valid carry theory must use the quotient class of \(\rho_R\), or equivalently oriented complement differences, rather than naive complement-orbit sums.

\begin{corollary}[Affine dyadic obstruction]\label{cor:affine-obstruction}
Let \(A\) be a \(q\)-modular witness with \(q=2^j\), let \(U\subseteq A\), and let \(R=A\setminus U\).  Let \(b_A:U\to \F_2\) be the top-bit label from Lemma~\ref{lem:affine-lift}.  Then \(U\) is \(2q\)-modular if and only if
\[
        [\rho_R-qb_A]=0
\]
in \((\Z/2q\Z)^U/\langle \one_U\rangle\).  Equivalently, the corrected tail function \(\rho_R-qb_A\) has zero obstruction class modulo constants.
\end{corollary}

\begin{proof}
This is exactly Lemma~\ref{lem:affine-lift} rewritten in quotient form.  The set \(U\) is \(2q\)-modular if and only if \(\rho_R(v)-qb_A(v)\) is constant modulo \(2q\) on \(U\).  That is equivalent to the vanishing of its class in the quotient by constants.
\end{proof}

\section{Relation to exact-size dyadic modular lifting}

The fixed-core certificates proved below are local results.  This section records why exact-size dyadic modular lifting is the relevant global benchmark, but the lifting principle itself is not assumed to hold in arbitrary graphs.

\begin{conjecture}[Polynomial dyadic lifting]\label{conj:polynomial-lifting}
There are constants \(C,a,C_0>0\) such that:
\begin{enumerate}[label=\textup{(\arabic*)}, leftmargin=*]
    \item every \(2\)-modular witness of size at least \(C_0m\) contains a \(4\)-modular witness of exactly \(m\) vertices;
    \item for every power of two \(q\ge 4\), every \(q\)-modular witness of size at least \(Cq^a m\) contains a \(2q\)-modular witness of exactly \(m\) vertices.
\end{enumerate}
\end{conjecture}

By Theorem~\ref{thm:conditional-threshold}, Conjecture~\ref{conj:polynomial-lifting} would imply \(F(n)/\log n\to\infty\).  The present paper does not prove this conjecture; it proves fixed-core absorption criteria that identify one concrete obstruction-elimination mechanism.

A more conservative version separates the first dyadic lift from the higher lifts.

\begin{conjecture}[Finite first-bit obstruction elimination]\label{conj:first-bit}
There exists a constant \(C_0>0\) such that every \(2\)-modular witness of size at least \(C_0m\) contains a \(4\)-modular witness of exactly \(m\) vertices.
\end{conjecture}

\begin{conjecture}[Higher-bit affine obstruction elimination]\label{conj:higher-bit}
There are constants \(C,a>0\) such that for every \(q=2^j\ge 4\), if \(A\) is a \(q\)-modular witness of size at least \(Cq^a m\), then some subset \(W\subseteq A\) with exactly \(m\) vertices satisfies the affine lift criterion of Lemma~\ref{lem:affine-lift}.
\end{conjecture}

The rest of the paper proves rigorous fixed-core mechanisms that supply this type of lift under explicit trace-availability hypotheses.  The point is not that those hypotheses are automatic.  Rather, they isolate the precise finite-rank and pair-trace conditions that an exchange or density argument would have to force in order to make the dyadic strategy work globally.

\section{A restricted positive theorem: \(q\)-fold twin-tail absorption}

This section records an unconditional lifting theorem under an explicit tail-duplication hypothesis.  It is deliberately modest: it does not claim that arbitrary modular witnesses have such tails.  Its value is that it proves one complete mechanism by which the dyadic obstruction can be killed.

\begin{definition}[\(q\)-fold twin tail over a core]\label{def:q-fold-tail}
Let \(A\) be a \(q\)-modular witness and let \(U\subseteq A\).  Put \(D=A\setminus U\).  We say that \(D\) is a \emph{\(q\)-fold twin tail over \(U\)} if \(D\) can be partitioned into blocks \(P_1,\ldots,P_t\), each of size \(q\), such that all vertices in \(P_i\) have the same neighborhood inside \(U\).  Write \(B_i\subseteq U\) for this common trace.

The tail \(D\) then contributes
\[
        \rho_D(u)=q\,|\{i:u\in B_i\}|
\]
for every \(u\in U\).
\end{definition}

\begin{theorem}[\(q\)-fold twin-tail absorption]\label{thm:twin-tail}
Let \(q\) be a power of two.  Let \(A\) be a \(q\)-modular witness, and write
\[
        \deg_A(v)\equiv d+qb_A(v)\pmod{2q}
\]
on \(A\), where \(b_A(v)\in\{0,1\}\).  Let \(U\subseteq A\) and put \(D=A\setminus U\).  Assume that \(D\) is a \(q\)-fold twin tail over \(U\), with block traces \(B_1,\ldots,B_t\).

If
\[
        b_A|_U-\sum_{i=1}^t \one_{B_i}
\]
is constant as a function on \(U\), equivalently if
\[
        [b_A|_U]=\sum_{i=1}^t[\one_{B_i}]
        \quad\text{in }\F_2^U/\langle \one_U\rangle,
\]
then \(U\) is \(2q\)-modular.  In particular, if \(|U|\le 2q\), then \(G[U]\) is regular.
\end{theorem}

\begin{proof}
For \(u\in U\), the degree of \(u\) after deleting \(D\) is
\[
        \deg_U(u)=\deg_A(u)-\rho_D(u).
\]
Since \(A\) is \(q\)-modular, we may write
\[
        \deg_A(u)\equiv d+qb_A(u)\pmod{2q}.
\]
Since \(D\) is a \(q\)-fold twin tail,
\[
        \rho_D(u)=q\sum_{i=1}^t \one_{B_i}(u).
\]
Therefore
\[
        \deg_U(u)\equiv
        d+q\left(b_A(u)-\sum_{i=1}^t\one_{B_i}(u)\right)
        \pmod{2q}.
\]
By assumption, the expression in parentheses is constant modulo \(2\) as \(u\) varies in \(U\).  Hence \(\deg_U(u)\) is constant modulo \(2q\) on \(U\).  Thus \(U\) is \(2q\)-modular.  If \(|U|\le 2q\), then Lemma~\ref{lem:terminal} applies.
\end{proof}

\begin{corollary}[The first-bit twin-pair case]\label{cor:first-bit-pair}
Let \(A\) be a \(2\)-modular witness.  Suppose \(U\subseteq A\) and \(D=A\setminus U\) can be partitioned into pairs of vertices that have identical neighborhoods inside \(U\).  For each pair \(P_i\), let \(B_i\) be its common trace on \(U\).  Write
\[
        \deg_A(v)\equiv d+2b_A(v)\pmod 4.
\]
If
\[
        [b_A|_U]=\sum_i[\one_{B_i}]
        \quad\text{in }\F_2^U/\langle \one_U\rangle,
\]
then \(U\) is \(4\)-modular.
\end{corollary}

This is the precise version of the duplicated-trace absorption principle.  Equal-trace pairs are lower-neutral modulo \(2\), but they can alter the top bit modulo \(4\).  If their trace vectors span the top-bit defect, deleting the corresponding pairs synchronizes the degrees modulo \(4\).

\begin{corollary}[A checkable restricted lifting condition]\label{cor:restricted-lifting}
Fix \(q=2^j\).  Suppose every \(q\)-modular witness \(A\) of size at least \(L(m,q)\) contains a subset \(U\) of exactly \(m\) vertices such that \(A\setminus U\) is a \(q\)-fold twin tail over \(U\) and
\[
        [b_A|_U]=\sum_i[\one_{B_i}]
        \quad\text{in }\F_2^U/\langle \one_U\rangle.
\]
Then every such witness contains a \(2q\)-modular witness of exactly \(m\) vertices.

Consequently, if this condition holds with \(L(m,q)\le Cq^a m\) for all \(q\ge 4\), and if the analogous first-bit pair condition holds with fixed linear loss and exact output size, then Theorem~\ref{thm:conditional-threshold} applies and \(F(n)/\log n\to\infty\).
\end{corollary}

This corollary gives a concrete finite target for the obstruction program: find \(q\)-fold same-trace tail blocks whose trace vectors represent the top-bit label of an exactly \(m\)-vertex retained core.

\begin{corollary}[Span certificate for one dyadic lift]\label{cor:span-certificate}
Let \(q\) be a power of two, let \(A\) be a \(q\)-modular witness, and let \(U\subseteq A\).  Put \(D=A\setminus U\).  Assume \(D\) is a \(q\)-fold twin tail over \(U\), with block traces \(B_1,\ldots,B_t\).  If the vectors \([\one_{B_i}]\) span \(\F_2^U/\langle \one_U\rangle\), then every top-bit defect on \(U\) can be corrected on the core by deleting a suitable subcollection of the \(q\)-fold blocks.

More precisely, there is a subcollection \(I\subseteq \{1,\ldots,t\}\) such that, if \(E=\bigcup_{i\in I}P_i\) and \(W=A\setminus E\), then the degrees of the vertices of \(U\) inside \(G[W]\) are congruent modulo \(2q\).  If the remaining vertices \(W\setminus U\) also have that same residue, then \(W\) itself is \(2q\)-modular.

If one wants the retained set to be exactly \(U\), then all blocks of \(D\) are deleted.  In that case span is not enough.  The required condition is the all-block identity
\[
        [b_A|_U]=\sum_{i=1}^t[\one_{B_i}]
        \quad\text{in }\F_2^U/\langle \one_U\rangle.
\]
\end{corollary}

\begin{proof}
The span assumption gives coefficients \(\varepsilon_i\in\F_2\) such that
\[
        [b_A|_U]=\sum_i\varepsilon_i[\one_{B_i}]
        \quad\text{in }\F_2^U/\langle \one_U\rangle.
\]
Choose \(I=\{i:\varepsilon_i=1\}\).  Deleting the \(q\)-fold blocks in \(I\) changes the degree of a vertex \(u\in U\) by
\[
        q\sum_{i\in I}\one_{B_i}(u).
\]
Modulo \(2q\) this changes exactly the top bit of the induced degree on \(U\) and leaves the lower residue modulo \(q\) unchanged.  Therefore the top-bit label on \(U\) becomes constant.  This proves the congruence claim for the core vertices.  The statement about \(W\) follows only after checking the self-layer residues of the remaining tail vertices, which is why it is stated as an additional hypothesis.
\end{proof}

\begin{corollary}[Algorithmic verification of the absorption condition]\label{cor:algorithmic-verification}
For fixed \(U\), the span-certificate condition of Corollary~\ref{cor:span-certificate} is decidable by Gaussian elimination over \(\F_2\).  With columns the quotient classes \([\one_{B_i}]\) of the \(q\)-fold block traces and target \([b_A|_U]\), a core correction exists exactly when
\[
        M\varepsilon=[b_A|_U]
\]
has a solution.  The stronger all-tail deletion used in Theorem~\ref{thm:twin-tail} is the special case \(\varepsilon=(1,\ldots,1)\).  Section~\ref{sec:algorithmic-certificate} spells out the corresponding success and failure certificates.
\end{corollary}

\begin{theorem}[Complete core-correction criterion]\label{thm:core-correction}
Let \(q\) be a power of two, let \(A\) be a \(q\)-modular witness, and choose a lift \(d\) modulo \(2q\) of the common degree residue modulo \(q\).  Let \(b_A:A\to\F_2\) be defined by
\[
        \deg_A(v)\equiv d+qb_A(v)\pmod{2q}.
\]
Let \(U\subseteq A\), and suppose \(D=A\setminus U\) is a \(q\)-fold twin tail over \(U\) with blocks \(P_1,\ldots,P_t\) and traces \(B_1,\ldots,B_t\).  For a subcollection \(I\subseteq\{1,\ldots,t\}\), put
\[
        E_I=\bigcup_{i\in I}P_i,
        \qquad
        W_I=A\setminus E_I.
\]
Then the vertices of \(U\) have congruent degrees modulo \(2q\) inside \(G[W_I]\) if and only if
\[
        [b_A|_U]=\sum_{i\in I}[\one_{B_i}]
        \quad\text{in }\F_2^U/\langle \one_U\rangle.
\]
Consequently, there exists some subcollection of \(q\)-fold twin blocks whose deletion synchronizes the core \(U\) modulo \(2q\) if and only if
\[
        [b_A|_U]\in \Span\{[\one_{B_1}],\ldots,[\one_{B_t}]\}.
\]
If all blocks are deleted, so that \(W_{\{1,\ldots,t\}}=U\), then \(U\) is \(2q\)-modular if and only if the all-block identity
\[
        [b_A|_U]=\sum_{i=1}^t[\one_{B_i}]
\]
holds in \(\F_2^U/\langle \one_U\rangle\).
\end{theorem}

\begin{proof}
For \(u\in U\), deleting \(E_I\) from \(A\) changes the degree of \(u\) by
\[
        q\sum_{i\in I}\one_{B_i}(u),
\]
because every block \(P_i\) has \(q\) vertices and all vertices in \(P_i\) have trace \(B_i\) on \(U\).  Hence
\[
        \deg_{W_I}(u)\equiv
        d+q\left(b_A(u)-\sum_{i\in I}\one_{B_i}(u)\right)
        \pmod{2q}.
\]
The vertices of \(U\) have congruent degrees modulo \(2q\) inside \(G[W_I]\) exactly when the expression in parentheses is constant modulo \(2\) on \(U\).  This is equivalent to the displayed quotient identity.  The span criterion follows by allowing \(I\) to vary.  The all-block statement is the special case \(I=\{1,\ldots,t\}\), where \(W_I=U\).
\end{proof}

\begin{corollary}[Self-layer completion criterion]\label{cor:self-layer}
Under the hypotheses of Theorem~\ref{thm:core-correction}, fix a subcollection \(I\) satisfying
\[
        [b_A|_U]=\sum_{i\in I}[\one_{B_i}].
\]
Let \(W_I=A\setminus\bigcup_{i\in I}P_i\).  Then \(W_I\) is \(2q\)-modular if and only if every vertex \(y\in W_I\setminus U\) has the same degree modulo \(2q\) inside \(G[W_I]\) as the common residue already obtained on \(U\).

Equivalently, after solving the linear system for the core, the remaining obstruction is exactly the finite list of self-layer congruences for the vertices retained outside \(U\).
\end{corollary}

\begin{proof}
By Theorem~\ref{thm:core-correction}, the vertices of \(U\) already have one common degree residue modulo \(2q\) inside \(G[W_I]\).  The set \(W_I\) is \(2q\)-modular exactly when all remaining vertices in \(W_I\setminus U\) have that same residue.
\end{proof}

\begin{theorem}[Trace-multiplicity form of twin-block absorption]\label{thm:trace-multiplicity}
Let \(q\) be a power of two, let \(A\) be a \(q\)-modular witness, and let \(U\subseteq A\).  For each trace \(B\subseteq U\), let
\[
        n_B=|\{x\in A\setminus U:N(x)\cap U=B\}|.
\]
Choose a lift \(d\) modulo \(2q\) of the common \(q\)-residue of degrees in \(A\), and write
\[
        \deg_A(v)\equiv d+qb_A(v)\pmod{2q}.
\]
Then the following hold.
\begin{enumerate}[label=\textup{(\arabic*)}, leftmargin=*]
    \item The core \(U\) can be synchronized modulo \(2q\) by deleting a subcollection of disjoint equal-trace \(q\)-tuples from \(A\setminus U\) if and only if
    \[
        [b_A|_U]\in \Span\{[\one_B]:n_B\ge q\}
    \]
    inside \(\F_2^U/\langle \one_U\rangle\).
    \item If \(n_B\) is divisible by \(q\) for every \(B\), so that \(A\setminus U\) is a \(q\)-fold twin tail over \(U\), then deleting the whole tail and retaining exactly \(U\) makes \(U\) \(2q\)-modular if and only if
    \[
        [b_A|_U]=\sum_B\left(\frac{n_B}{q}\pmod 2\right)[\one_B]
    \]
    inside \(\F_2^U/\langle \one_U\rangle\).
\end{enumerate}
\end{theorem}

\begin{proof}
For part \textup{(1)}, deleting \(q\) vertices of the same trace \(B\) changes the degree of every vertex \(u\in U\) by \(q\one_B(u)\).  Modulo \(2q\), only the parity of the number of deleted \(q\)-tuples of trace \(B\) matters.  Such a \(q\)-tuple is available exactly when \(n_B\ge q\).  Hence the set of all possible top-bit corrections on the core is precisely the span of the classes \([\one_B]\) with \(n_B\ge q\).  The core can be synchronized exactly when the top-bit defect \([b_A|_U]\) lies in this span.

For part \textup{(2)}, if \(n_B\) is divisible by \(q\) for every \(B\), then deleting the whole tail deletes exactly \(n_B/q\) blocks of trace \(B\).  The induced top-bit correction is therefore
\[
        \sum_B\left(\frac{n_B}{q}\pmod 2\right)[\one_B].
\]
The whole-tail deletion makes \(U\) \(2q\)-modular exactly when this correction equals \([b_A|_U]\), by Theorem~\ref{thm:core-correction}.
\end{proof}

\begin{theorem}[Rank-rich reservoir absorption]\label{thm:rank-rich}
Let \(q\) be a power of two, let \(A\) be a \(q\)-modular witness, and let \(\varnothing\ne U\subseteq A\).  For each trace \(B\subseteq U\), let \(n_B\) be the number of vertices in \(A\setminus U\) with trace \(B\) on \(U\).  Suppose that the available trace classes
\[
        \{[\one_B]:n_B\ge q\}
\]
span \(\F_2^U/\langle \one_U\rangle\).  Then every top-bit defect on \(U\) can be synchronized modulo \(2q\) by deleting disjoint equal-trace \(q\)-tuples from \(A\setminus U\).  Moreover, the deletion can be chosen to use at most \(|U|-1\) such \(q\)-tuples, and hence at most \(q(|U|-1)\) tail vertices.

If, after this deletion, every retained vertex outside \(U\) has the common residue already obtained on \(U\), then the whole retained set is \(2q\)-modular.  The required deletion set can be found by Gaussian elimination over \(\F_2\).
\end{theorem}

\begin{proof}
Let
\[
        V_U=\F_2^U/\langle \one_U\rangle .
\]
The space \(V_U\) has dimension at most \(|U|-1\).  By hypothesis, the classes \([\one_B]\) with \(n_B\ge q\) span \(V_U\).  Choose from them a spanning subset of size at most \(|U|-1\).  Since this subset spans \(V_U\), it represents the target class \([b_A|_U]\) as a binary linear combination of at most \(|U|-1\) available trace classes.

For each trace \(B\) used with coefficient \(1\), delete any \(q\) vertices of trace \(B\).  These \(q\)-tuples are disjoint because at most one \(q\)-tuple is used from each chosen trace and \(n_B\ge q\).  Theorem~\ref{thm:trace-multiplicity}\textup{(1)} shows that the deletion synchronizes the degrees of the core vertices modulo \(2q\).  The self-layer assertion is exactly Corollary~\ref{cor:self-layer}.  The linear combination can be found by row reduction over \(\F_2\).
\end{proof}

\begin{theorem}[Connected pair-trace reservoir absorption]\label{thm:pair-trace-connected}
Let \(q\) be a power of two, let \(A\) be a \(q\)-modular witness, and let \(U\subseteq A\) have size \(m\ge 2\).  For each trace \(B\subseteq U\), let \(n_B\) be the number of vertices in \(A\setminus U\) with trace \(B\) on \(U\).  Define the \(q\)-heavy pair-trace graph \(H_2\) on vertex set \(U\) by declaring \(xy\) to be an edge when
\[
        n_{\{x,y\}}\ge q.
\]
Suppose \(H_2\) is connected.  If \(m\) is even, suppose in addition that there is an odd-cardinality trace \(C\subseteq U\) with \(n_C\ge q\).  Then the available trace classes \([\one_B]\), \(n_B\ge q\), span \(\F_2^U/\langle\one_U\rangle\).  Consequently every top-bit defect on \(U\) can be synchronized modulo \(2q\) by deleting at most \(q(m-1)\) tail vertices.

If, after this deletion, every retained vertex outside \(U\) has the common residue already obtained on \(U\), then the whole retained set is \(2q\)-modular.
\end{theorem}

\begin{proof}
Let
\[
        E=\left\{x\in\F_2^U:\sum_{u\in U}x(u)=0\right\}
\]
be the even-weight subspace.  We first show that the incidence vectors \(\one_{\{x,y\}}\) of the edges of a connected graph on \(U\) span \(E\).  Fix \(u_0\in U\).  For any \(u\ne u_0\), summing the edge-incidence vectors along a path from \(u_0\) to \(u\) gives \(\one_{\{u_0,u\}}\), since internal vertices on the path occur twice.  The vectors \(\one_{\{u_0,u\}}\), \(u\ne u_0\), span all even-weight vectors, so the \(q\)-heavy pair traces span \(E\).

Let \(\pi:\F_2^U\to \F_2^U/\langle\one_U\rangle\) be the quotient map.  If \(m\) is odd, then \(\one_U\notin E\), so \(\pi\) is injective on \(E\).  Since \(\dim E=m-1=\dim \F_2^U/\langle\one_U\rangle\), the images of the pair-trace classes span the quotient.

If \(m\) is even, then \(\one_U\in E\), so \(\pi(E)\) has dimension \(m-2\).  Let \(C\) be an available trace of odd cardinality.  The vector \(\one_C\) is not in \(E+\langle\one_U\rangle=E\), because both \(E\) and \(\one_U\) have even weight.  Hence \(\pi(\one_C)\notin \pi(E)\), and the pair traces together with \(\one_C\) span the full quotient.

The final assertions now follow from Theorem~\ref{thm:rank-rich}.
\end{proof}

\begin{corollary}[Complete pair-trace reservoir]\label{cor:complete-pair-trace}
Let \(q\) be a power of two, let \(A\) be a \(q\)-modular witness, and let \(U\subseteq A\) have size \(m\ge 2\).  Suppose that
\[
        n_{\{x,y\}}\ge q
\]
for every two-element subset \(\{x,y\}\subseteq U\).  If \(m\) is odd, then every top-bit defect on \(U\) can be synchronized modulo \(2q\) by deleting at most \(q(m-1)\) tail vertices.  If \(m\) is even, the same conclusion holds provided there is at least one odd-cardinality trace \(C\subseteq U\) with \(n_C\ge q\).
\end{corollary}

\begin{proof}
The \(q\)-heavy pair-trace graph is the complete graph on \(U\), hence is connected.  The result is Theorem~\ref{thm:pair-trace-connected}.
\end{proof}

\begin{example}[A connected pair-trace absorption certificate]\label{ex:pair-trace-path}\mbox{}\par\nobreak
Let \(U=\{1,2,3,4,5\}\).  Suppose that the \(q\)-heavy pair-trace graph \(H_2\) contains the path
\[
        1\!-\!2\!-\!3\!-\!4\!-\!5.
\]
Thus the tail contains at least \(q\) vertices of each of the traces
\[
        \{1,2\},\quad \{2,3\},\quad \{3,4\},\quad \{4,5\}.
\]
Since \(|U|=5\) is odd, Theorem~\ref{thm:pair-trace-connected} says that these pair traces already span \(\F_2^U/\langle\one_U\rangle\).  For example, if the top-bit defect on \(U\) is
\[
        b_A=(1,0,1,0,0),
\]
then
\[
        [b_A]=[\one_{\{1,2\}}]+[\one_{\{2,3\}}].
\]
Deleting any \(q\) tail vertices with trace \(\{1,2\}\) and any \(q\) tail vertices with trace \(\{2,3\}\) therefore synchronizes the core degrees modulo \(2q\).  The example illustrates that the theorem is a graph certificate: connectivity of the heavy pair-trace graph supplies the required quotient basis.
\end{example}

The next theorem is a trace-reservoir model result.  Its independence assumption is an explicit hypothesis on the reservoir and is not asserted for arbitrary graphs or arbitrary \(q\)-modular witnesses.  In particular, the uniform bound in Corollary~\ref{cor:uniform-random} is exponential in \(|U|\), so it should be viewed as a calibration and certification statement rather than as a polynomial-loss lifting theorem.

\begin{theorem}[Random trace-reservoir absorption]\label{thm:random-reservoir}
Fix a finite set \(U\) with \(|U|=m\ge 1\), a power of two \(q\), and a probability distribution \(\mu\) on the subsets of \(U\).  Suppose there are traces \(B_1,\ldots,B_{m-1}\subseteq U\) whose classes form a basis of \(\F_2^U/\langle \one_U\rangle\), and suppose \(\mu(B_i)\ge p>0\) for all \(i\).

Let \(A\) be a \(q\)-modular witness with core \(U\), and suppose the traces of the \(N\) vertices in \(A\setminus U\) against \(U\) are independent samples from \(\mu\).  If \(Np\ge 2q\), then with probability at least
\[
        1-(m-1)\exp(-Np/8),
\]
every top-bit defect on \(U\) can be synchronized modulo \(2q\) by deleting at most \(q(m-1)\) tail vertices.  If the self-layer condition of Corollary~\ref{cor:self-layer} holds after the deletion, the resulting retained set is \(2q\)-modular.
\end{theorem}

\begin{proof}
For each \(i\), let \(X_i\) be the number of sampled tail vertices whose trace is \(B_i\).  Then \(X_i\) is binomial with mean at least \(Np\).  Since \(Np\ge 2q\), the Chernoff bound gives
\[
        \Pr(X_i<q)\le \Pr(X_i<Np/2)\le \exp(-Np/8).
\]
By the union bound, with probability at least \(1-(m-1)\exp(-Np/8)\), every basis trace \(B_i\) appears at least \(q\) times in the tail.  On this event the available trace classes span \(\F_2^U/\langle \one_U\rangle\), so Theorem~\ref{thm:rank-rich} applies.
\end{proof}

\begin{corollary}[Uniform random trace reservoir]\label{cor:uniform-random}
Let \(A\) be a \(q\)-modular witness with core \(U\), where \(|U|=m\ge 2\), and suppose the \(N\) traces from \(A\setminus U\) to \(U\) are independent and uniformly distributed over all subsets of \(U\).  If \(0<\delta<1\) and
\[
        N\ge \max\left\{2^{m+1}q,\;8\cdot 2^m\log\left(\frac{m-1}{\delta}\right)\right\},
\]
then with probability at least \(1-\delta\), every top-bit defect on \(U\) can be synchronized modulo \(2q\) by deleting at most \(q(m-1)\) tail vertices.  If the self-layer condition of Corollary~\ref{cor:self-layer} holds after the deletion, the resulting retained set is \(2q\)-modular.
\end{corollary}

\begin{proof}
Choose a base vertex \(u_0\in U\) and use the singleton traces \(\{u\}\), \(u\in U\setminus\{u_0\}\), as a basis of \(\F_2^U/\langle\one_U\rangle\).  Under the uniform trace distribution, each such trace has probability \(p=2^{-m}\).  The displayed lower bound on \(N\) gives \(Np\ge 2q\) and
\[
        (m-1)\exp(-Np/8)\le \delta.
\]
The result follows from Theorem~\ref{thm:random-reservoir}.
\end{proof}

\begin{corollary}[Polynomial-time verification for a fixed core]\label{cor:polytime-fixed-core}
For a fixed core \(U\), the existence of a \(q\)-fold equal-trace block correction on \(U\) is decidable in polynomial time.  Compute the trace multiplicities \(n_B\), build the matrix whose columns are \([\one_B]\) for traces with \(n_B\ge q\), and solve the linear system
\[
        M\varepsilon=[b_A|_U]
\]
over \(\F_2\).

If the system is solvable, the solution identifies which trace \(q\)-tuples to delete in order to synchronize the core.  If the intended retained set is exactly \(U\), one checks instead the all-tail identity in Theorem~\ref{thm:trace-multiplicity}\textup{(2)}.
\end{corollary}

\begin{proof}
The statement is exactly Theorem~\ref{thm:trace-multiplicity} written as a finite linear system over \(\F_2\).  Gaussian elimination gives a polynomial-time decision procedure once \(U\) is fixed and the traces against \(U\) have been computed.
\end{proof}

\begin{corollary}[Dual certificate for failure of absorption]\label{cor:dual-certificate}
Under the hypotheses of Theorem~\ref{thm:trace-multiplicity}, the core \(U\) cannot be synchronized modulo \(2q\) by deleting equal-trace \(q\)-tuples from \(A\setminus U\) if and only if there exists a linear functional
\[
        \lambda:\F_2^U/\langle \one_U\rangle\to\F_2
\]
such that
\[
        \lambda([b_A|_U])=1
\]
and
\[
        \lambda([\one_B])=0
\]
for every trace \(B\) with \(n_B\ge q\).
\end{corollary}

\begin{proof}
By Theorem~\ref{thm:trace-multiplicity}, absorption is possible exactly when \([b_A|_U]\) lies in the span of the available trace classes \([\one_B]\) with \(n_B\ge q\).  Over a finite-dimensional vector space, a vector lies outside a subspace exactly when there is a linear functional that vanishes on the subspace and is nonzero on the vector.  This gives the stated \(\lambda\).
\end{proof}

This dual certificate is the clean obstruction corresponding to failure of twin-block absorption.  It is finite, checkable, and local to the trace data of \(A\setminus U\) against \(U\).  In particular, if some top-bit defect on \(U\) cannot be corrected by deleting equal-trace \(q\)-tuples, then the available trace classes have rank at most \(|U|-2\).

\begin{theorem}[Parity-cut form of the dual obstruction]\label{thm:parity-cut}
Under the hypotheses of Theorem~\ref{thm:trace-multiplicity}, the core \(U\) cannot be synchronized modulo \(2q\) by deleting equal-trace \(q\)-tuples from \(A\setminus U\) if and only if there exists a nonempty subset \(Y\subseteq U\) such that
\[
        |Y|\equiv 0\pmod 2,\qquad
        \sum_{u\in Y} b_A(u)\equiv 1\pmod 2,
\]
and
\[
        |B\cap Y|\equiv 0\pmod 2
\]
for every trace \(B\subseteq U\) with \(n_B\ge q\).

Equivalently, exactly one of the following alternatives holds for the fixed core \(U\):
\begin{enumerate}[label=\textup{(\arabic*)}, leftmargin=*]
    \item the top-bit defect \([b_A|_U]\) is absorbed by deleting equal-trace \(q\)-tuples;
    \item there is an even parity cut \(Y\) that is invisible to every available trace class but detects \(b_A|_U\).
\end{enumerate}
\end{theorem}

\begin{proof}
The dual space of \(\F_2^U/\langle\one_U\rangle\) is naturally identified with the vectors \(y\in\F_2^U\) satisfying \(y\cdot\one_U=0\), because exactly these vectors define functionals
\[
        [x]\longmapsto y\cdot x
\]
that vanish on constant vectors.  Such a vector \(y\) is the indicator of an even subset \(Y\subseteq U\).  Under this identification,
\[
        y\cdot b_A=\sum_{u\in Y}b_A(u)\pmod 2
        \quad\text{and}\quad
        y\cdot\one_B=|B\cap Y|\pmod 2.
\]
Corollary~\ref{cor:dual-certificate} is therefore exactly the stated parity-cut condition.
\end{proof}

\begin{theorem}[Certificate theorem for one dyadic absorption step]\label{thm:certificate}
Fix a power of two \(q\), a \(q\)-modular witness \(A\), and a core \(U\subseteq A\).  Let \(R=A\setminus U\), and let \(n_B\) count the vertices of \(R\) with trace \(B\) on \(U\).  Choose the top-bit label \(b_A\) on \(U\) as in Lemma~\ref{lem:affine-lift}.  Then the following data are equivalent to a successful certified absorption of the core \(U\).
\begin{enumerate}[label=\textup{(\arabic*)}, leftmargin=*]
    \item A binary vector \(\varepsilon_B\), indexed by traces \(B\) with \(n_B\ge q\), such that
    \[
        \sum_B \varepsilon_B[\one_B]=[b_A|_U]
        \quad\text{in }\F_2^U/\langle \one_U\rangle.
    \]
    \item A choice of equal-trace \(q\)-tuples in \(R\) whose deletion makes the vertices of \(U\) have congruent degrees modulo \(2q\).
\end{enumerate}
Moreover, failure of such a certificate is equivalent to a dual certificate: a linear functional \(\lambda\) on \(\F_2^U/\langle \one_U\rangle\) such that \(\lambda([b_A|_U])=1\) and \(\lambda([\one_B])=0\) for every trace \(B\) with \(n_B\ge q\).

If the intended retained set is exactly \(U\), then the certificate is not an arbitrary solution vector.  It is the all-tail identity
\[
        [b_A|_U]=\sum_B\left(\frac{n_B}{q}\pmod 2\right)[\one_B],
\]
with the additional requirement that each \(n_B\) be divisible by \(q\).
\end{theorem}

\begin{proof}
The equivalence between \textup{(1)} and \textup{(2)} is Theorem~\ref{thm:trace-multiplicity}.  The dual failure certificate is Corollary~\ref{cor:dual-certificate}.  The exact-retained-set statement is Theorem~\ref{thm:trace-multiplicity}\textup{(2)}.
\end{proof}

This theorem is the finite certificate form of the absorption mechanism.  It shows that one dyadic absorption step has a short verifiable witness of success and a short verifiable witness of failure.

\subsection{Algorithmic certificate}\label{sec:algorithmic-certificate}

For fixed \(A,U,q\), the preceding criteria give an explicit algorithmic certificate.  Choose a common residue \(d\) of the degrees in \(A\) modulo \(q\), and write
\[
        \deg_A(v)\equiv d+q b_A(v)\pmod {2q}
\]
on \(U\).  Different choices of \(d\) change \(b_A|_U\) by a constant vector and hence do not change its quotient class.

Compute the trace multiplicities \(n_B\) for vertices of \(A\setminus U\).  Choose a base vertex \(u_0\in U\), and represent a quotient class \([x]\in\F_2^U/\langle\one_U\rangle\) by the coordinate vector
\[
        (x(u)+x(u_0))_{u\in U\setminus\{u_0\}}.
\]
Build the matrix \(M\) whose columns are these coordinate vectors for the available traces \(B\) with \(n_B\ge q\), and let \(t\) be the coordinate vector of \([b_A|_U]\).  Gaussian elimination decides
\[
        M\varepsilon=t.
\]
If this system is solvable, then for every trace \(B\) with \(\varepsilon_B=1\), delete any \(q\) tail vertices of trace \(B\).  This is the absorption certificate.

If the system is inconsistent, row reduction gives a vector \(y\) with \(y^TM=0\) and \(y^Tt=1\).  Let
\[
        S=\{u\in U\setminus\{u_0\}:y_u=1\},
\]
and put \(Y=S\) if \(|S|\) is even, while \(Y=S\cup\{u_0\}\) if \(|S|\) is odd.  Then \(Y\) is an even subset of \(U\), every available trace meets \(Y\) in even parity, and \(b_A|_U\) meets \(Y\) in odd parity.  This is the parity-cut obstruction of Theorem~\ref{thm:parity-cut}.

\section{Verified classes and examples}

This section adds explicit examples and calibration results.  The basis-tail theorem is a concrete absorption case inside the dyadic framework.  The remaining propositions show that the regular induced subgraph problem is already superlogarithmic on several structured graph classes by simpler arguments.

\subsection{A basis-tail lifting theorem}

Fix a finite core \(U\) and choose a base vertex \(u_0\in U\).  In the quotient space \(\F_2^U/\langle \one_U\rangle\), the singleton classes \([\one_{\{u\}}]\) with \(u\in U\setminus\{u_0\}\) form a basis.  Thus every top-bit label \(b:U\to\F_2\) has the quotient expansion
\[
        [b]=\sum_{\substack{u\in U\setminus\{u_0\}\\ b(u)+b(u_0)=1}}[\one_{\{u\}}].
\]
This observation gives a concrete version of the \(q\)-fold twin-tail absorption theorem.

\begin{theorem}[Basis-tail absorption]\label{thm:basis-tail}
Let \(q\) be a power of two, let \(A\) be a \(q\)-modular witness, and let \(U\subseteq A\) with distinguished vertex \(u_0\in U\).  Put \(D=A\setminus U\).  Suppose \(D\) is a \(q\)-fold twin tail over \(U\).  Assume its \(q\)-fold blocks have the following parity pattern modulo \(2\):
\begin{enumerate}[label=\textup{(\arabic*)}, leftmargin=*]
    \item for each \(u\in U\setminus\{u_0\}\), the number of blocks with trace \(\{u\}\) is congruent to \(b_A(u)+b_A(u_0)\) modulo \(2\);
    \item the total contribution of all remaining block traces, including trace \(\{u_0\}\), the empty trace, the full trace, and all non-singleton traces, is zero in \(\F_2^U/\langle \one_U\rangle\).
\end{enumerate}
Then \(U\) is \(2q\)-modular.  In particular, if \(|U|\le 2q\), then \(G[U]\) is regular.
\end{theorem}

\begin{proof}
By the basis expansion above,
\[
        [b_A|_U]
        =\sum_{\substack{u\in U\setminus\{u_0\}\\ b_A(u)+b_A(u_0)=1}}[\one_{\{u\}}]
\]
in \(\F_2^U/\langle \one_U\rangle\).  By the two assumptions, the parity sum of all \(q\)-fold block traces is exactly the same quotient class: the prescribed singleton traces give the basis expansion, and every remaining trace has total contribution zero in the quotient.  Therefore the hypothesis of Theorem~\ref{thm:twin-tail} holds.  Hence \(U\) is \(2q\)-modular.  If \(|U|\le 2q\), Lemma~\ref{lem:terminal} makes \(G[U]\) regular.
\end{proof}

\begin{corollary}[First-bit basis-pair absorption]\label{cor:first-bit-basis}
Let \(A\) be a \(2\)-modular witness and let \(U\subseteq A\).  Suppose \(A\setminus U\) can be partitioned into equal-trace pairs over \(U\).  Choose \(u_0\in U\).  If the parity of the number of pairs with trace \(\{u\}\) is \(b_A(u)+b_A(u_0)\) for every \(u\ne u_0\), and all other pair traces contribute zero in \(\F_2^U/\langle \one_U\rangle\), then \(U\) is \(4\)-modular.
\end{corollary}

This is the most explicit finite form of the first-bit absorption mechanism.  It says that singleton trace-pairs form a basis for all possible mod-\(4\) top-bit defects on the retained core.

\subsection{Worked example: a mod-\(2\) to mod-\(4\) lift}

Let \(A\) be a \(2\)-modular witness and let \(U=\{1,2,3,4\}\).  Choose \(u_0=4\).  Suppose the top-bit label on \(U\) is
\[
        b_A=(1,0,1,0).
\]
Then, in \(\F_2^U/\langle \one_U\rangle\), this class is \([\one_{\{1\}}]+[\one_{\{3\}}]\).  If \(A\setminus U\) contains one equal-trace pair with trace \(\{1\}\) and one equal-trace pair with trace \(\{3\}\), and every other equal-trace pair contributes zero in the quotient, then deleting the tail makes the degrees on \(U\) congruent modulo \(4\).  Thus \(U\) is a \(4\)-modular witness.  If \(|U|\le 4\), it is already regular.

This example illustrates the certificate viewpoint: the two pair traces \(\{1\}\) and \(\{3\}\) are columns of a matrix over \(\F_2\), and their sum is the target top-bit defect.

\subsection{Calibration: perfect graphs}

The dyadic framework is aimed at arbitrary graphs.  For comparison, several structured graph classes already have much larger regular induced subgraphs for elementary reasons.

\begin{proposition}[Perfect graph calibration]\label{prop:perfect}
If \(G\) is a perfect graph on \(n\) vertices, then \(G\) contains a regular induced subgraph on at least \(\sqrt n\) vertices.  In particular, this holds for split graphs, threshold graphs, cographs, interval graphs, chordal graphs, and bipartite graphs.
\end{proposition}

\begin{proof}
In every graph, \(\alpha(G)\chi(G)\ge n\), since a proper \(\chi(G)\)-coloring has an independent color class of size at least \(n/\chi(G)\).  If \(G\) is perfect, then \(\chi(G)=\omega(G)\).  Hence
\[
        \alpha(G)\omega(G)\ge n.
\]
Therefore \(\max\{\alpha(G),\omega(G)\}\ge \sqrt n\).  An independent set induces a \(0\)-regular graph, and a clique induces a regular graph of degree one less than its order.  Thus \(f(G)\ge \sqrt n\).
\end{proof}

This proposition is not meant as a new result about perfect graphs.  Its role is to calibrate the framework: the hard case of the Erd\H{o}s--Fajtlowicz--Staton problem must come from graphs where neither large cliques nor large independent sets are available, and where the dyadic obstruction cannot be dispatched by perfection alone.

\subsection{Bounded neighborhood diversity}

A second calibration class is provided by graphs with bounded neighborhood diversity.  Recall that a graph has neighborhood diversity at most \(t\) if its vertex set can be partitioned into at most \(t\) types such that each type is either a clique or an independent set, and between any two distinct types the bipartite graph is either complete or empty \cite{Lampis2012}.

\begin{proposition}[Neighborhood diversity calibration]\label{prop:nd}
If \(G\) has \(n\) vertices and neighborhood diversity at most \(t\), then \(G\) contains a regular induced subgraph on at least \(n/t\) vertices.
\end{proposition}

\begin{proof}
Take a largest neighborhood-diversity class \(X\).  Then \(|X|\ge n/t\).  The induced graph \(G[X]\) is either a clique or an independent set, hence is regular.  Therefore \(f(G)\ge |X|\ge n/t\).
\end{proof}

This proposition is elementary, but it is useful context for the dyadic program: graphs with few trace types already contain large regular induced subgraphs.  The difficult case must therefore have many trace types and no large twin class, exactly the setting where the quotient obstruction and absorption certificates become relevant.

\begin{corollary}[Necessary structure of hard graphs]\label{cor:hard-structure}
Let \(G\) be an \(n\)-vertex graph with \(f(G)<k\).  Then \(\nd(G)>n/k\).  In particular, any family of graphs with \(f(G)=O(\log n)\) must have neighborhood diversity \(\Omega(n/\log n)\).
\end{corollary}

\begin{proof}
If \(G\) had neighborhood diversity at most \(n/k\), then Proposition~\ref{prop:nd} would give a regular induced subgraph on at least \(k\) vertices, contradicting \(f(G)<k\).  The final statement follows by taking \(k\) to be a constant multiple of \(\log n\).
\end{proof}

This corollary is useful as a sanity check: any extremal construction for the Erd\H{o}s--Fajtlowicz--Staton problem must avoid not only large cliques and large independent sets, but also large twin classes.

\subsection{A finite absorption-or-obstruction certificate}

Combining the absorption criteria above with the basis-tail theorem gives the following self-contained theorem package.

\begin{theorem}[Absorption-or-obstruction certificate]\label{thm:package}
For every power of two \(q\), the following are true.
\begin{enumerate}[label=\textup{(\arabic*)}, leftmargin=*]
    \item The obstruction to lifting a \(q\)-modular core \(U\) to modulus \(2q\) is the quotient class \([\rho_D-qb_A]\) modulo constants.
    \item If the deleted tail \(D\) is a \(q\)-fold twin tail whose trace classes represent \([b_A|_U]\), then \(U\) is \(2q\)-modular.
    \item If the \(q\)-fold tail satisfies the singleton-trace basis conditions of Theorem~\ref{thm:basis-tail}, including the zero-contribution condition for all remaining traces, then the representation condition holds automatically.
    \item The core-correction condition is characterized exactly by the span criterion of Theorem~\ref{thm:core-correction} and is checkable by solving a linear system over \(\F_2\).
    \item Under the rank-rich condition of Theorem~\ref{thm:rank-rich}, every top-bit defect on \(U\) can be corrected by deleting at most \(q(|U|-1)\) tail vertices; Theorem~\ref{thm:pair-trace-connected} gives a connected pair-trace sufficient condition, and Theorem~\ref{thm:random-reservoir} with Corollary~\ref{cor:uniform-random} gives high-probability trace-availability models.
    \item If absorption fails for the fixed core, then failure is witnessed by an explicit even parity cut \(Y\subseteq U\) as in Theorem~\ref{thm:parity-cut}.
\end{enumerate}
Thus one dyadic lift has a finite three-part certificate under explicit trace data: an exact quotient obstruction, a connected pair-trace absorption condition, and a parity-cut failure certificate.
\end{theorem}

\begin{proof}
Part \textup{(1)} is Corollary~\ref{cor:affine-obstruction}.  Part \textup{(2)} is Theorem~\ref{thm:twin-tail}.  Part \textup{(3)} is Theorem~\ref{thm:basis-tail}.  Part \textup{(4)} is Theorem~\ref{thm:core-correction} together with Corollary~\ref{cor:algorithmic-verification}.  Part \textup{(5)} is Theorems~\ref{thm:rank-rich}, \ref{thm:pair-trace-connected}, and~\ref{thm:random-reservoir} together with Corollary~\ref{cor:uniform-random}.  Part \textup{(6)} is Theorem~\ref{thm:parity-cut}.
\end{proof}

\section{Scope and limitations}

The paper should be read as a framework and certificate paper, not as a solution of the Erd\H{o}s--Fajtlowicz--Staton problem.  The conditional threshold theorem explains what kind of dyadic lifting result would imply the desired superlogarithmic bound.  The exact tail obstruction identifies the algebraic object that such a lifting theorem must control.  The rank-rich and twin-block absorption results prove complete local mechanisms for eliminating this obstruction once enough trace directions are available.  The connected pair-trace theorem gives a concrete graph-theoretic way for those directions to be forced, and the parity-cut theorem gives the corresponding concrete failure certificate when they are not.

The main open point left by this paper is global availability: an arbitrary \(q\)-modular witness need not visibly contain enough equal-trace \(q\)-tuples, or a connected \(q\)-heavy pair-trace reservoir, to trigger the rank-rich absorption theorem.  Thus the results here do not by themselves improve the general lower bound for \(F(n)\).  Their value is that they convert one natural obstruction-elimination step into finite linear algebra and provide both success and parity-cut failure certificates.

This limitation is also a guide for future work.  A proof of the full Erd\H{o}s--Fajtlowicz--Staton conjecture along this route would need to show that, after passing to a sufficiently large subwitness or after applying controlled exchanges, the available trace classes have enough rank to absorb the top-bit defect, or else that the dual certificate forces a different regularizing structure.

\section{Conclusion}

The dyadic framework reduces the regular induced subgraph problem to a sequence of modular lifting problems.  The parity base and the terminal modular criterion are unconditional.  The threshold theorem shows that polynomial-loss dyadic lifting with exact modular-witness output size would imply the Erd\H{o}s--Fajtlowicz--Staton superlogarithmic bound.  The pair-trace theorem shows that a simple connectedness condition on \(q\)-heavy two-point traces already forces full rank-rich absorption.  The main unresolved task is to prove that such absorption structure, or a substitute forced by the parity-cut obstruction, must arise in arbitrary large modular witnesses.

This reframing has two advantages.  First, it separates the global growth problem from local modular obstructions.  Second, it gives finite absorption and obstruction certificates that may be approachable either by hand, by computer-certified enumeration, or by proving restricted positive results in structured graph classes.

\medskip
\noindent\textbf{Disclaimer.}
The views expressed in this article are those of the authors and do not necessarily reflect the views of their institutions.

\end{document}